\documentclass[11pt,reqno]{amsart}
\calclayout
\usepackage{amssymb,mathtools,xcolor}
\usepackage[hidelinks]{hyperref}
\newtheorem{Thm}{Theorem}[section]
\newtheorem{Lem}[Thm]{Lemma}
\newtheorem{Prop}[Thm]{Proposition}
\newtheorem{Cor}[Thm]{Corollary}

\DeclareFontEncoding{OT2}{}{}
\DeclareFontSubstitution{OT2}{cmr}{m}{n}
\DeclareFontFamily{OT2}{cmr}{\hyphenchar\font45 }
\DeclareFontShape{OT2}{cmr}{m}{n}{
<5><6><7><8><9>gen*wncyr
<10><10.95><12><14.4><17.28><20.74><24.88>wncyr10}{}
\DeclareFontShape{OT2}{cmr}{b}{n}{
<5><6><7><8><9>gen*wncyb
<10><10.95><12><14.4><17.28><20.74><24.88>wncyb10}{}
\DeclareMathAlphabet{\mathcyr}{OT2}{cmr}{m}{n}
\DeclareMathAlphabet{\mathcyb}{OT2}{cmr}{b}{n}
\SetMathAlphabet{\mathcyr}{bold}{OT2}{cmr}{b}{n}

\newcommand{\bk}{\boldsymbol{k}}
\newcommand{\FF}{\mathbb{F}}

\newcommand{\ZZ}{\mathbb{Z}}
\newcommand{\QQ}{\mathbb{Q}}
\newcommand{\HH}{\mathfrak{H}}

\newcommand{\calA}{\mathcal{A}}
\newcommand{\calF}{\mathcal{F}}

\newcommand{\calS}{\mathcal{S}}
\newcommand{\calZ}{\mathcal{Z}}
\newcommand{\sh}{\mathbin{\mathcyr{sh}}}
\newcommand{\Ker}{\mathrm{Ker}}

\begin{document}
\title[A generating function identity for finite and symmetric multiple zeta values]{A generating function identity for finite and symmetric multiple zeta values via Dougall's bilateral ${}_2H_2$ summation}
\author{Takumi Maesaka, Masataka Ono}
\address{Faculty of Mathematics, Kyushu University, 744 Motooka, Nishi-ku, Fukuoka 819-0395, Japan}
\email{nozaki.takumi.912@s.kyushu-u.ac.jp}
\address{School of Business Administration, Senshu University, 744 Higashimita, Tama-ku, Kawasaki, Kanagawa 214-8580, Japan.}
\email{ono@isc.senshu-u.ac.jp}
\begin{abstract}

We establish a generating function identity for finite and symmetric multiple zeta values using Dougall's bilateral ${}_2H_2$ summation formula. As an application, we prove a restricted sum formula in which the parity of each component of the index is prescribed. 
We also derive a generating function identity for finite multiple zeta values $\zeta^{}_{\calA}(\{1,2\}^n)$ and symmetric multiple zeta values $\zeta^{}_{\calS}(\{1,2\}^n)$, thereby proving Kaneko's conjecture for finite multiple zeta values and its symmetric analogue. 
\end{abstract}
\maketitle

\section{Introduction}

For an index $\bk=(k_1,\dots,k_r)$ with $k_r\geq 2$, the multiple zeta value (MZV) $\zeta(\bk)$ is defined by
\[
\zeta(\bk)=\sum_{0<n_1<\cdots<n_r}\frac 1{n_1^{k_1}\cdots n_r^{k_r}}.
\]
Let $\calZ$ denote the $\QQ$-algebra generated by all MZVs.

For an index $\bk=(k_1,\dots,k_r)$, let
\[
\zeta^*(\bk;T)\in\calZ[T]
\]
denote its harmonic regularization; see Ihara, Kaneko and Zagier \cite{ikz}. We write
\[
\zeta^*(\bk)\coloneqq \zeta^*(\bk;0)
\]
and set $\zeta^*(\varnothing)\coloneqq1$. Following Kaneko and Zagier \cite{kz}, for an index $\bk=(k_1,\dots,k_r)$ we define the symmetric multiple zeta value by
\[
\zeta^{}_{\calS}(\bk)\coloneqq\sum_{i=0}^r(-1)^{k_{i+1}+\cdots+k_r}\zeta^*(k_1,\dots,k_i)\zeta^*(k_r,\dots,k_{i+1})\pmod{\pi^2\calZ},
\]
which is an element of the $\QQ$-algebra $\calZ_{\calS}:=\calZ/\pi^2\calZ$. Here, an empty sequence of arguments is interpreted as the empty index.

Finite multiple zeta values provide another analogue of MZVs. We define the $\QQ$-algebra $\calA$ by
\[
\calA\coloneqq\left(\prod_{p:\mathrm{prime}}\ZZ/p\ZZ\right)\bigg/\left(\bigoplus_{p:\mathrm{prime}}\ZZ/p\ZZ\right).
\]
Thus, two prime-indexed sequences represent the same element of $\calA$ if their components agree for all but finitely many primes. For an index $\bk=(k_1,\ldots,k_r)$, its finite multiple zeta value is defined by
\[
\zeta^{}_{\calA}(\bk)\coloneqq\left(\sum_{0<n_1<\cdots<n_r<p}\frac 1{n_1^{k_1}\cdots n_r^{k_r}}\mod p\right)_p\in \calA.
\]
We also set $\zeta^{}_{\calA}(\varnothing)\coloneqq 1$ and denote by $\calZ_{\calA}$ the $\QQ$-subalgebra of $\calA$ generated by all finite multiple zeta values. Kaneko and Zagier \cite{kz} conjectured that the correspondence
\[
\zeta^{}_{\calA}(\bk)\mapsto \zeta^{}_{\calS}(\bk)
\]
extends to an isomorphism of $\QQ$-algebras
\[
\calZ_{\calA}\simeq \calZ_{\calS}.
\]
This conjecture suggests that finite and symmetric multiple zeta values should satisfy parallel families of algebraic relations.


The sum formula for finite multiple zeta values was proved by Saito and Wakabayashi \cite{sw}, and its symmetric counterpart was established by Murahara \cite{m}. Subsequently, Murahara and Saito \cite{ms} proved restricted sum formulas for both finite and symmetric multiple zeta values. Theorem \ref{thm:sub} below concerns a different restriction: one prescribed component of the index is even and all the remaining components are odd.

Hypergeometric identities provide a useful approach to generating functions for sums of multiple zeta values. Ohno and Zagier \cite{oz} obtained a generating function for sums of multiple zeta values of fixed weight, depth, and height in terms of Riemann zeta values, using Gauss’s hypergeometric summation formula and the gamma function. Kawamura and the authors \cite{kmo} established weighted sum formulas for half multiple zeta values and their finite and symmetric analogues. Their proofs for half multiple zeta values and half symmetric multiple zeta values use Dougall’s ${}_5F_4$ summation formula.

In this paper, we establish an identity of generating function for finite and symmetric multiple zeta values in four formal parameters subject to the relation $ab=cd$ (Theorem \ref{thm:main}). For symmetric multiple zeta values, we derive the identity from Dougall’s bilateral ${}_2H_2$ summation formula by harmonic regularization. For finite multiple zeta values, we give an independent algebraic proof in characteristic $p$, based on a multiplicative 2-cocycle identity for a truncated hypergeometric series. The specializations of the resulting generating function yield the parity-restricted sums in Theorem \ref{thm:sub} and prove Kaneko’s conjecture \cite{k} on the generating series of $\zeta^{}_{\calA}(\{1,2\}^{n})$, together with its symmetric counterpart (Theorem \ref{thm:one-two}). Here, for $n\geq 0$ and an index $\bk=(k_1, \ldots, k_r)$, we write $\{k_1, \ldots, k_r\}^n$ for the $n$-fold concatenation of the index $\bk$, with $\{k_1, \ldots, k_r\}^0\coloneqq\varnothing$.

For a complex number $a$ and a nonnegative integer $n$, define
\begin{align*}
(a)_n&\coloneqq \begin{cases}
1&n=0,\\
a(a+1)\cdots(a+n-1)&n>0,
\end{cases}\\
(a)_{-n}&\coloneqq \frac{(-1)^n}{(1-a)_n}.
\end{align*}
For complex numbers $a_1,\dots,a_r$ and an integer $n$, we further write
\[
(a_1,\dots,a_r)_n\coloneqq (a_1)_n\cdots(a_r)_n.
\]
The main ingredient in our proofs is Dougall's bilateral ${}_2H_2$ summation formula.
\begin{Thm}[Dougall \cite{d}]\label{thm:dougall}
Let $a,b,c,d$ be complex numbers satisfying $\Re(c+d-a-b)>1$, and assume that both sides of the following identity are well-defined. Then
\[
\sum_{n\in\ZZ}\frac{(a,b)_n}{(c,d)_n}=\frac{\Gamma(c)\Gamma(d)\Gamma(1-a)\Gamma(1-b)\Gamma(c+d-a-b-1)}{\Gamma(c-a)\Gamma(d-a)\Gamma(c-b)\Gamma(d-b)}
\]
holds.
\end{Thm}
Let $\HH\coloneqq \QQ\langle x,y\rangle$ be the noncommutative polynomial algebra over $\QQ$ in the variables $x$ and $y$, and put $ \HH^1\coloneqq \QQ+y\HH$. Whenever $F$ is a function from the set of indices (including the empty index) to a $\QQ$-vector space $V$, we also use $F$ to denote the unique $\QQ$-linear map $F:\HH^1\to V$ determined by
\[
F(1_{\HH})\coloneqq F(\varnothing),\qquad F(yx^{k_1-1}\cdots yx^{k_r-1})\coloneqq F(k_1,\dots,k_r).
\]
We shall use this convention without further comment.

We now state our main result.
For $\calF \in \{\calA, \calS\}$ and a positive integer $n \ge2$, put
\[
Z_{\calF}(n)
\coloneqq 
\begin{cases}
\left(\frac{B_{p-n}}{n} \bmod{p}\right)_p \in \calA & \text{if $\calF=\calA$}, \\
\zeta(n) \bmod{\pi^2\calZ} & \text{if $\calF=\calS$}.
\end{cases}
\]
Here, $B_n \in \QQ$ is the Bernoulli number defined by the series $\frac{x}{e^x-1}=\sum_{n=0}^{\infty}\frac{B_n}{n!}x^n$. Note that the element $Z_{\mathcal{A}}(k)$ may be regarded as a finite analogue of the Riemann zeta value $\zeta(k)$.

\begin{Thm}\label{thm:main}
The following identity holds in $\calZ_{\calF}[[a,b,c,d]]/(ab-cd)$:
\begin{multline*}
1+ab\zeta^{}_{\calF}\left(\frac 1{1-y(a+b)-yxab}yx\frac 1{1+y(c+d)-yxcd}\right)\\
=\exp\bigg(\sum_{2\leq n}\frac{Z_{\calF}(n)}{n}(a^n+b^n+c^n+d^n+(a+b+c+d)^n\\
\qquad\qquad\qquad -(a+c)^n-(a+d)^n-(b+c)^n-(b+d)^n)\bigg).
\end{multline*}
\end{Thm}

By specializing the variables and comparing the coefficients, we obtain the following theorem.
\begin{Thm}\label{thm:sub}
Let $k$ and $r$ be positive integers, and let $i$ be an integer with $1\leq i\leq r$. Then we have
\begin{align*}
\sum_{\substack{1\leq k_1,\dots,k_r\\k_1,\dots,k_{i-1},k_{i+1},\dots,k_r:\mathrm{odd}\\k_i:\mathrm{even}\\k_1+\cdots+k_r=k}}\zeta^{}_{\calF}(k_1, \ldots, k_r)\in\QQ[Z_{\calF}(3), Z_{\calF}(5),\dots].
\end{align*}
\end{Thm}
Note that Kaneko, Murakami and Yoshihara \cite{kmy} considered a similar restrected sum formula for finite multiple zeta values of level two. Its restredction is that one prescribed component of the index is odd and all the remaining components are even. They gave a conjecture on that restricted sum formula, which was recently solved by Zhenlu Wang in \cite{w}.

As another application of Theorem \ref{thm:main}, we obtain the following theorem, which was conjectured by Kaneko \cite{k} for the case $\calF=\calA$. 
\begin{Thm}\label{thm:one-two}
The following identity holds in $\calZ_{\calF}[[t]]$:
\[
\sum_{0\leq n}\zeta^{}_{\calF}(\{1,2\}^n)t^{n}=\exp\left(\sum_{0<n}\zeta^{}_{\calF}(1,3n-1)\frac{t^{n}}{n^2}\right).
\]
\end{Thm}

The remainder of this paper is organized as follows. In Section \ref{sec:mainS}, we prove Theorem \ref{thm:main} with $\calF=\calS$, and we prove Theorem \ref{thm:main} with $\calF=\calA$ in Section \ref{sec:mainA}. In Section \ref{sec:restricted sum formula}, we prove Theorem \ref{thm:sub}. In Section \ref{sec:one-two}, we derive Theorem \ref{thm:one-two} from Theorem \ref{thm:main} using the shuffle and harmonic relations for finite and symmetric multiple zeta values.

\section{Proof of Theorem \ref{thm:main} with $\calF=\calS$}\label{sec:mainS}

For integers $0\leq M<N$ and a nonempty index $\bk=(k_1,\dots,k_r)$, define
\begin{align*}
\zeta^{}_{M,N}(\bk)&\coloneqq \sum_{M<n_1<\cdots<n_r<N}\frac {1}{n_1^{k_1}\cdots n_r^{k_r}},
\end{align*}
and set $\zeta^{}_{M,N}(\varnothing)\coloneqq 1$. For any index $\bk$, put $\zeta^{}_N(\bk)\coloneqq \zeta^{}_{0,N}(\bk)$.

To prove Theorem \ref{thm:main}, we need the following lemma.
\begin{Lem}\label{lem:truncated}
Let $N$ be a positive integer. The following identity of formal power series in $a,b,c,d$ holds:
\[
\sum_{n=1}^{N-1}\frac{(a,b)_n}{(1-c,1-d)_n}=ab\zeta^{}_{N}\left(\frac 1{1-y(a+b)-yxab}yx\frac 1{1+y(c+d)-yxcd}\right)\frac{(N-1)!^2}{(1-c,1-d)_{N-1}}.
\]
\end{Lem}

\begin{proof}
A direct calculation gives
\begin{align*}
&\sum_{n=1}^{N-1}\frac{(a,b)_n}{(1-c,1-d)_n}\frac{(1-c,1-d)_{N-1}}{(N-1)!^2}\\
&=ab\sum_{n=1}^{N-1}\left(\prod_{j=1}^{n-1}\left(1+\frac{a+b}j+\frac{ab}{j^2}\right)\right)\frac 1{n^2}\prod_{l=n+1}^{N-1}\left(1-\frac{c+d}l+\frac{cd}{l^2}\right)\\
&=ab\sum_{n=1}^{N-1}\zeta^{}_{N}\left(\frac 1{1-y(a+b)-yxab}\right)\frac 1{n^2}\zeta^{}_{n,N}\left(\frac 1{1+y(c+d)-yxcd}\right)\\
&=ab\zeta^{}_{N}\left(\frac 1{1-y(a+b)-yxab}yx\frac 1{1+y(c+d)-yxcd}\right).
\end{align*}
This completes the proof.
\end{proof}
The identities
\[
\frac{(1+X)_{N-1}}{(N-1)!}=\sum_{0\leq n}\zeta^{}_{N}(\{1\}^n)X^n
\]
and
\[
\sum_{0\leq n}\zeta^*(\{1\}^n)X^n=\frac 1{\Gamma^*(1+X)}\coloneqq\frac 1{e^{\gamma X}\Gamma(1+X)}
\]
hold, where $\gamma$ denotes the Euler--Mascheroni constant. Taking the limit $N\to\infty$ in Lemma \ref{lem:truncated} and applying harmonic regularization, we obtain the following corollary.
\begin{Cor}\label{cor:infinite}
The following identity of formal power series in $a,b,c,d$ holds:
\[
\sum_{0<n}\frac{(a,b)_n}{(1-c,1-d)_n}=ab\zeta^*\left(\frac 1{1-y(a+b)-yxab}yx\frac 1{1+y(c+d)-yxcd}\right)\Gamma^*(1-c)\Gamma^*(1-d).
\]
\end{Cor}

\begin{proof}[Proof of Theorem \ref{thm:main} with $\calF=\calS$]
Define the $\QQ$-linear signed reversal map $\rho:\HH^1\to\HH^1$ by
\[
\rho(1_{\HH})=1_{\HH},\qquad
\rho(yx^{k_1-1}\cdots yx^{k_r-1})
=(-1)^{k_1+\cdots+k_r}
yx^{k_r-1}\cdots yx^{k_1-1}.
\]
We extend $\rho$ coefficientwise to $\HH^1[[a,b,c,d]]$. Put
\[
P=\frac{1}{1-y(a+b)-yxab},
\qquad
Q=\frac{1}{1+y(c+d)-yxcd}.
\]
Expanding $P$ and $Q$ as formal power series and splitting the defining sum for $\zeta^{}_{\calS}(PyxQ)$ according to whether the cut occurs before or after the distinguished factor $yx$, we obtain
\[
\zeta^{}_{\calS}(PyxQ)
=\zeta^*(PyxQ)\zeta^*(\rho(Q))
 +\zeta^*(P)\zeta^*(\rho(PyxQ)).
\]
Since
\[
\rho(Q)=\frac{1}{1-y(c+d)-yxcd}
\]
and
\[
\rho(PyxQ)
=\frac{1}{1-y(c+d)-yxcd}
 yx\frac{1}{1+y(a+b)-yxab},
\]
multiplying the preceding identity by $ab$ and using $ab=cd$ in the second term gives

\begin{align*}
&ab\zeta^{}_{\calS}\left(\frac 1{1-y(a+b)-yxab}yx\frac 1{1+y(c+d)-yxcd}\right)\\
&=ab\zeta^*\left(\frac 1{1-y(a+b)-yxab}yx\frac 1{1+y(c+d)-yxcd}\right)\zeta^*\left(\frac 1{1-y(c+d)-yxcd}\right)\\
&\qquad+cd\zeta^*\left(\frac 1{1-y(a+b)-yxab}\right)\zeta^*\left(\frac 1{1-y(c+d)-yxcd}yx\frac 1{1+y(a+b)-yxab}\right).
\end{align*}
Since
\begin{align*}
\zeta^{}_{N}\left(\frac 1{1-y(a+b)-yxab}\right)&=\prod_{n=1}^{N-1}\left(1+\frac{a+b}{n}+\frac{ab}{n^2}\right)\\
&=\frac{(1+a,1+b)_{N-1}}{(N-1)!^2},
\end{align*}
harmonic regularization yields
\begin{align*}
\zeta^*\left(\frac 1{1-y(a+b)-yxab}\right)&=\frac{1}{\Gamma^*(1+a)\Gamma^*(1+b)}.
\end{align*}
Hence, combining the identity above with Corollary \ref{cor:infinite} yields
\begin{align*}
&ab\zeta^{}_{\calS}\left(\frac 1{1-y(a+b)-yxab}yx\frac 1{1+y(c+d)-yxcd}\right)\\
&=\frac 1{\Gamma^*(1+c)\Gamma^*(1-c)\Gamma^*(1+d)\Gamma^*(1-d)}\sum_{0<n}\frac{(a,b)_n}{(1-c,1-d)_n}\\
&\qquad+\frac 1{\Gamma^*(1+a)\Gamma^*(1-a)\Gamma^*(1+b)\Gamma^*(1-b)}\sum_{0<n}\frac{(c,d)_n}{(1-a,1-b)_n}.
\end{align*}
In $\calZ_{\calS}$,
\[
\frac 1{\Gamma^*(1+X)\Gamma^*(1-X)}=\frac{\sin\pi X}{\pi X}\equiv 1\pmod{\pi^2\calZ}.
\]
Hence, we obtain
\begin{align*}
&ab\zeta^{}_{\calS}\left(\frac 1{1-y(a+b)-yxab}yx\frac 1{1+y(c+d)-yxcd}\right)\\
&=\sum_{0<n}\frac{(a,b)_n}{(1-c,1-d)_n}+\sum_{0<n}\frac{(c,d)_n}{(1-a,1-b)_n}\\
&=\sum_{n\in\ZZ}\frac{(a,b)_n}{(1-c,1-d)_n}-1\\
&=\frac{\Gamma(1-a)\Gamma(1-b)\Gamma(1-c)\Gamma(1-d)\Gamma(1-a-b-c-d)}{\Gamma(1-a-c)\Gamma(1-a-d)\Gamma(1-b-c)\Gamma(1-b-d)}-1.
\end{align*}
The last equality follows from Theorem \ref{thm:dougall}. Applying the identity 
\begin{align*}
    \Gamma(1-x)=\exp\left(\gamma x+\sum_{2\leq n}\frac{\zeta(n)}{n}x^n\right)
\end{align*}
to the above expression yields the desired result.
\end{proof}

\section{Proof of Theorem \ref{thm:main} with $\calF=\calA$}\label{sec:mainA}
In this section, we prove Theorem \ref{thm:main}. Applying Lemma \ref{lem:truncated}, we can rewrite Theorem \ref{thm:main} as follows.
\begin{Thm}\label{thm:rewrite_mainA}
    The following identity holds in $\calZ_{\calA}[[a,b,c,d]]/(ab-cd)\calZ_{\calA}[[a,b,c,d]]$:
\begin{align*}
\left(\sum_{n=0}^{p-1}\frac{(a,b)_n}{(1-c,1-d)_n}\mod p\right)_p&=\exp\bigg(\sum_{2\leq n}\frac{Z_{\mathcal{A}}(n)}{n}(a^n+b^n+c^n+d^n+(a+b+c+d)^n\\
&\qquad\qquad\qquad -(a+c)^n-(a+d)^n-(b+c)^n-(b+d)^n)\bigg).
\end{align*}
\end{Thm}

Fix a positive integer $N$ and a prime $p$ satisfying $p>N+2$. Put
\begin{align*}
F_p(a,b,c,d) &\coloneqq \sum_{n=0}^{p-1}\frac{(a,b)_n}{(1-c,1-d)_n}
\end{align*}
and
\begin{align*}
Q_n(a,b,c,d) &\coloneqq a^n+b^n+c^n+d^n+(a+b+c+d)^n\\
&\qquad\qquad\qquad -(a+c)^n-(a+d)^n-(b+c)^n-(b+d)^n.
\end{align*}
To prove Theorem \ref{thm:main} with $\calF=\calA$, it is enough to show the equality
\begin{align}\label{eq:enough_to_show}
F_p(a,b,c,d)=\exp\left(\sum_{n=2}^{N}\frac{B_{p-n}}{n^2}Q_n(a,b,c,d)\right)
\end{align}
in $R_{p,N}\coloneqq \FF_p[a,b,c,d]/(ab-cd,(a,b,c,d)^{N+1})$.
Set $a=xu, b=yv, c=xv, d=yu$. Since $ab-cd=0$, we can define a natural $\FF_p$-algebra homomorphism $\varphi\colon\FF_p[[a,b,c,d]]/(ab-cd)\to \FF_p[[x,y,u,v]]$, which is injective. Moreover, since $\varphi((a,b,c,d)^{N+1}) \subset (x,y)^{N+1}$, a well-defined injective homomorphism $\overline{\varphi}_N\colon R_{p,N}\to \FF_p[u,v][x,y]/(x,y)^{N+1}$ is induced. Put 
\begin{align*}
\overline{G}_p(x,y;u,v)
&\coloneqq F_p(xu,yv,xv,yu)\bmod{(x,y)^{N+1}} \\
&=
\sum_{n=0}^{p-1}\frac{(xu)_n(yv)_n}{(1-xv)_n(1-yu)_n} \bmod{(x,y)^{N+1}}\in \FF_p[u,v][x,y]/(x,y)^{N+1}.
\end{align*}

\begin{Prop}\label{prop:G_p 1-cocyle}
We obtain the following equality
\begin{align}\label{eq:G_p 1-cocyle}
\overline{G}_p(x,y;u,v)\overline{G}_p(x+y,z;u,v)=\overline{G}_p(y,z;u,v)\overline{G}_p(x,y+z;u,v).
\end{align}
\end{Prop}

\begin{proof}
Set $\eta\coloneqq-\frac{v}{u}, E\coloneqq \FF_p(u,v)[t]/(t^p-\eta)$ and define a $\FF_p(u,v)$-linear map $\theta\colon E\to E$ by $\theta \coloneqq t\frac{d}{dt}$. Since $\{1, t, \ldots, t^{p-1}\}$ is a $\FF_p(u,v)$-basis of $E$ and $\theta(t^j)=jt^j \ (1\le j \le p-1)$, we have $\Ker\,\theta=\FF_p(u,v)$.
For an indeterminate $X$ and a non-negative integer $0 \le n \le p-1$, set
\begin{align*}
h_n(X)
\coloneqq \frac{(Xu)_n}{(1-Xv)_n} \in \FF_p(u,v)(X), \quad
H_X(t)\coloneqq \sum_{n=0}^{p-1}h_n(X)t^{n} \in \FF_p(u,v)(X)[t].
\end{align*}
Then we have
\begin{align*}
(1-t)\theta H_X(t)-X(v+ut)H_X(t)
&=\sum_{n=0}^{p-1}(n-Xv)h_n(X)t^n-\sum_{n=0}^{p-1}(n+Xu)h_n(X)t^{n+1}\\
&=-Xv-(p-1+Xu)h_{p-1}(X)t^p.
\end{align*}
Moreover, we have
\begin{align*}
(p-1+Xu)h_{p-1}(X)
&=\frac{(Xu)_p}{(1-Xv)_{p-1}}
=\frac{Xu-X^pu^p}{1-X^{p-1}v^{p-1}}\\
&\equiv Xu \bmod{X^p}. 
\end{align*}
Since $t^p=\eta=-v/u$ and $1-t$ is invertible in $E$, we have 
\begin{align*}
\theta H_X(t)\equiv XR(t)H_X(t) \bmod{X^p},
\end{align*}
where $R(t)\coloneqq\frac{v+ut}{1-t}$. Then, putting $X=x,y,x+y$, we have the following equalities in $E[x,y]/(x,y)^p$:
\begin{align*}
\theta H_x(t)\equiv xR(t)H_x(t),\quad \theta H_y(t)\equiv yR(t)H_y(t), \quad
\theta H_{x+y}(t)\equiv (x+y)R(t)H_{x+y}(t).
\end{align*}
By differentiating above, we have
\begin{align*}
\theta\left(\frac{H_x(t)H_y(t)}{H_{x+y}(t)}\right)\equiv 0 \bmod{(x,y)^p}.
\end{align*}
From above, there exists $C_p(x,y) \in E[x,y]/(x,y)^{N+1}$ satisfying 
\begin{align}\label{eq:H_xH_y=CH_{x+y}}
H_x(t)H_y(t)=C_p(x,y)H_{x+y}(t).
\end{align}
By comparing the constant term in \eqref{eq:H_xH_y=CH_{x+y}} and using $t^p=\eta=-v/u$, we have
\begin{align*}
C_p(x,y)=1+\eta\sum_{n=1}^{p-1}h_n(x)h_{p-n}(y).
\end{align*}
Moreover, since
\begin{align*}
\eta h_{p-n}(y)
=\frac{1-y^{p-1}u^{p-1}}{1-y^{p-1}v^{p-1}}\frac{(yv)_n}{(1-yu)_n}
\equiv \frac{(yv)_n}{(1-yu)_n} \bmod{y^p},
\end{align*}
we see that
\begin{align}\label{eq:C_p=G_p}
C_p(x,y)=\overline{G}_p(x,y;u,v).
\end{align}
Substituting \eqref{eq:C_p=G_p} into \eqref{eq:H_xH_y=CH_{x+y}}, we have 
\begin{align}\label{eq:H_xH_y=G_pH_{x+y}}
H_x(t)H_y(t)=\overline{G}_p(x,y;u,v)H_{x+y}(t).
\end{align}
Therefore, using \eqref{eq:H_xH_y=G_pH_{x+y}} to rewrite
\begin{align*}
    (H_x(t)H_y(t))H_z(t)=H_x(t)(H_y(t)H_z(t))
\end{align*}
into $H_{x+y+z}(t)$, we complete the proof.
\end{proof}
Next, put
\begin{align*}
\log \overline{G}_p
\coloneqq \sum_{r=1}^{N}\frac{(-1)^{r+1}}{r}(\overline{G}_p-1)^r \in \FF_p[u,v][x,y]/(x,y)^{N+1}.
\end{align*}
Since $\overline{G}_p-1 \in \FF_p[u,v][x,y]/(x,y)^{N+1}$ is a polynomial of degree $N$, $\log \overline{G}_p$ is also a polynomial of degree $N$. Let $L_n$ denote the degree $n$ part of $\log\overline{G}_p$, then we have $\log \overline{G}_p= \sum_{n=2}^{N}L_n$.

\begin{Prop}\label{prop:L_n}
We have
\begin{align*}
L_n(x,y;u,v)=\frac{B_{p-n}}{n^2}((x+y)^n-x^n-y^n)((u+v)^n-u^n-v^n).
\end{align*}
\end{Prop}

\begin{proof}
Taking $\log$ of both sides of \eqref{eq:G_p 1-cocyle} and comparing the degree $n$ part, we have
\begin{align*}
L_n(x,y;u,v)+L_n(x+y,z;u,v)=L_n(y,z;u,v)+L_n(x,y+z;u,v).
\end{align*}
Since $\overline{G}_p(x,0;u,v)=\overline{G}_p(0,y;u,v)=1$ by the definition of $\overline{G}_p$, we have $L_n(x,0;u,v)=L_n(0,y;u,v)=0$.
Let $\partial_2$ denote the differentiation with respect to the second variable. Since
\begin{align*}
\partial_2L_n(x+y,z;u,v)=\partial_2L_n(y,z;u,v)+\partial_2L_n(x,y+z;u,v),
\end{align*}
by substituting $z=0$, we have
\begin{align}\label{eqref:del2}
\partial_2L_n(x,y;u,v)=\partial_2L_n(x+y,0;u,v)-\partial_2L_n(y,0;u,v).
\end{align}
Since $L_n(x,y)$ is a homogeneous polynomial of degree $n$ with respect to $x,y$, we can write $L_n$ as follows:
\begin{align*}
L_n(x,y;u,v)=\sum_{j=0}^{n}A_j(u,v)x^{n-j}y^j,
\end{align*}
where $A_j(u,v) \in \FF_p[u,v]$. Thus, we have $\partial_2L_n(x,0;u,v)=A_1(u,v)x^{n-1}$. Substituting this equality into \eqref{eqref:del2} and rewriting $A_1(u,v)$ as $A(u,v)$, we have
\begin{align*}
\frac{\partial L_n}{\partial y}(x,y;u,v)=A(u,v)((x+y)^{n-1}-y^{n-1}).
\end{align*}
Therefore, since $n<p$, we have
\begin{align*}
L_n(x,y;u,v)=\frac{A(u,v)}{n}((x+y)^n-x^n-y^n).
\end{align*}
Let $[x^{k_1}_1\cdots x^{k_n}_n](f)$ be the coefficient of $x^{k_1}_1\cdots x^{k_n}_n$ in a polynomial $f(x_1, \ldots, x_n)$. Put $\lambda_{p,n}\coloneqq[x^{n-1}y](A(x,y)) \in \FF_p$. Since $L_n$ is a bi-homogeneous polynomial of degree $n$ with respect to $x,y$, and also to $u,v$ and $\overline{G}_p(x,y;u,v)=\overline{G}_p(u,v;x,y)$, we have
\begin{align}\label{eq:L_n(x,y;u,v)=L_n(u,v;x,y)}
L_n(x,y;u,v)=L_n(u,v;x,y).
\end{align}
Comparing the coefficient of $x^{n-1}y$ in \eqref{eq:L_n(x,y;u,v)=L_n(u,v;x,y)}, we see that
\begin{align*}
A(u,v)=\frac{\lambda_{p,n}}{n}((u+v)^n-u^n-v^n),
\end{align*}
which leads to
\begin{align*}
L_n(x,y;u,v)=\frac{\lambda_{p,n}}{n^2}((x+y)^n-x^n-y^n)((u+v)^n-u^n-v^n).
\end{align*}

Finally, we prove $\lambda_{p,n}=B_{p-n}$. By the expression above, we have $\lambda_{p,n}=[x^{n-1}yu^{n-1}v](L_n)=[x^{n-1}yu^{n-1}v](\log\overline{G}_p)$. Moreover, since $[x^{n-1}yu^{n-1}v](\overline{G}_p-1)^j=0$ for $j\ge2$, we see that $[x^{n-1}yu^{n-1}v](\log\overline{G}_p)=[x^{n-1}yu^{n-1}v](\overline{G}_p)$.
By the definition of $\overline{G}_p$, we have
\begin{align*}
[x^{n-1}yu^{n-1}v](\overline{G}_p)
&=[x^{n-1}yu^{n-1}v]\sum_{m=1}^{p-1}\frac{xyuv}{m^2}\frac{\prod_{j=1}^{m-1}(1+\frac{xu}{j})(1+\frac{yv}{j})}{\prod_{j=1}^{m}(1-\frac{xv}{j})(1-\frac{yu}{j})}\\
&=\sum_{m=1}^{p-1}\frac{1}{m^2}[x^{n-2}u^{n-2}]\prod_{j=1}^{m-1}\left(1+\frac{xu}{j}\right)\\
&=\zeta^{}_{p}(\{1\}^{n-2},2)\equiv B_{p-n} \pmod{p}.
\end{align*}
The last congruence comes from \cite[Theorem 4.5]{hht}.
This completes the proof.
\end{proof}
\begin{proof}[{Proof of Theorem \ref{thm:main}} with $\calF=\calA$]
Since
\begin{align*}
((x+y)^n-x^n-y^n)((u+v)^n-u^n-v^n)=Q_n(xu,yv,xv,yu)=\varphi(Q_n(a,b,c,d)),
\end{align*}
by Proposition \ref{prop:L_n}, we have
\begin{align*}
L_n(x,y;u,v)=\frac{B_{p-n}}{n^2}\varphi(Q_n(a,b,c,d)).
\end{align*}
By the injectivity of $\varphi$ and $\overline{\varphi}_N$, we see that 
\begin{align*}
F_{p}(a,b,c,d)=\exp\left(\sum_{n=2}^{N}\frac{B_{p-n}}{n^2}Q_n(a,b,c,d)\right),
\end{align*}
which is \eqref{eq:enough_to_show}. This completes the proof of Theorem \ref{thm:main} with $\calF=\calA$.
\end{proof}

\section{Restricted sum formula}\label{sec:restricted sum formula}
In this section, as an application of our main theorem, we prove Theorem \ref{thm:sub}.
\begin{proof}[Proof of Theorem \ref{thm:sub}]
For an indeterminate $s$, put
\begin{align*}
\alpha_{\pm}(s)\coloneqq \frac{s\pm\sqrt{s^2+4}}{2} \in \QQ[[s]]
\end{align*}
and for indeterminates $u,v, t$, put $(a,b,c,d)=(\alpha_+(u)t, \alpha_-(u)t, -\alpha_+(v)t,-\alpha_-(v)t)$. Then we have an injection
\begin{align*}
\calZ_{\calF}[[a,b,c,d]]/(ab-cd) \to \calZ_{\calF}[[u,v,t]].
\end{align*}
For a non-negative integer $n$, set
\begin{align*}
q_n(u,v) \coloneq Q_n(\alpha_+(u),\alpha_-(u),-\alpha_+(v),-\alpha_-(v)),
\end{align*}
which is an element in $\QQ[u,v]$ by the fundamental theorem of symmetric polynomials. From Theorem \ref{thm:main}, we have
\begin{align}\label{eq:Phi_calF}
\begin{split}
t^2\zeta^{}_{\calF}\biggl(\frac{1}{1-uty+t^2yx}yx\frac{1}{1-vty+t^2yx}\biggr)&=1-\exp\biggl(\sum_{2\le n}\frac{Z_{\calF}(n)}{n}q_n(u,v)t^n\biggr).
\end{split}
\end{align}

Let $*$ denote the harmonic product. A calculation similar to that in the proof of Lemma \ref{lem:truncated} gives
\begin{align*}
&\zeta^{}_{N}\left(\frac 1{1-uty+yxt^2}yx\frac 1{1-vty+yxt^2}\right)\\
&=\sum_{n=1}^{N-1}\left(\prod_{k=1}^{n-1}\left(1+\frac{ut}{k}-\frac{t^2}{k^2}\right)\right)\frac 1{n^2}\prod_{k=n+1}^{N-1}\left(1+\frac{vt}{k}-\frac{t^2}{k}\right)\\
&=\left(\prod_{n=1}^{N-1}\left(1-\frac{t^2}{n^2}\right)\right)\sum_{n=1}^{N-1}\left(\prod_{k=1}^{n-1}\left(1+\frac{utk}{k^2-t^2}\right)\right)\frac 1{n^2-t^2}\prod_{k=n+1}^{N-1}\left(1+\frac{vtk}{k^2-t^2}\right)\\
&=\zeta^{}_{N}\left(\frac 1{1+yxt^2}\right)\zeta^{}_{N}\left(\frac 1{1-uty\frac 1{1-x^2t^2}}yx\frac 1{1-x^2t^2}\frac 1{1-vty\frac 1{1-x^2t^2}}\right)\\
&=\zeta^{}_{N}\left(\frac 1{1+yxt^2}*\frac 1{1-uty\frac 1{1-x^2t^2}}yx\frac 1{1-x^2t^2}\frac 1{1-vty\frac 1{1-x^2t^2}}\right).
\end{align*}
Here, the last equality follows from the harmonic relation for finite multiple harmonic sums. By the linear independence of multiple harmonic sums (see Yamamoto \cite[Theorem 3.1]{y}), we have
\begin{align*}
\frac 1{1-uty+yxt^2}yx\frac 1{1-vty+yxt^2}=\frac 1{1+yxt^2}*\frac 1{1-uty\frac 1{1-x^2t^2}}yx\frac 1{1-x^2t^2}\frac 1{1-vty\frac 1{1-x^2t^2}}.
\end{align*}

Since $\zeta^{}_{\calF}(k)=0$ for every positive integer $k$, the harmonic relation  (see Kaneko \cite[(8.3) and Proposition 9.2]{k2}) gives
\[
\zeta^{}_{\calF}\left(\frac 1{1+yxt^2}\right)=1.
\]
Hence, we obtain
\begin{align*}
&t^2\zeta^{}_{\calF}\left(\frac 1{1-uty+yxt^2}yx\frac 1{1-vty+yxt^2}\right)\\
&=t^2\zeta^{}_{\calF}\left(\frac 1{1-uty\frac 1{1-x^2t^2}}yx\frac 1{1-x^2t^2}\frac 1{1-vty\frac 1{1-x^2t^2}}\right)\\
&=\sum_{1\leq i\leq r}u^{i-1}v^{r-i}t^{r+1}\zeta^{}_{\calF}\left(\left(y\frac 1{1-x^2t^2}\right)^{i-1}yx\frac 1{1-x^2t^2}\left(y\frac 1{1-x^2t^2}\right)^{r-i}\right)\\
&=\sum_{\substack{1\leq i\leq r<k\\r\not\equiv k\,\mathrm{mod}\,2}}u^{i-1}v^{r-i}t^{k}\sum_{\substack{1\leq k_1,\dots,k_r\\k_1,\dots,k_{i-1},k_{i+1},\dots,k_r:\mathrm{odd}\\k_i:\mathrm{even}\\k_1+\cdots+k_r=k}}\zeta^{}_{\calF}(k_1,\dots,k_r).
\end{align*}
Therefore, from \eqref{eq:Phi_calF}, we obtain the desired formula.
\end{proof}

\section{A proof of Kaneko's conjecture and its symmetric analogue} \label{sec:one-two}
In this section, we prove the Kaneko's conjecture and its symmetric analogue. We denote the shuffle product and the harmonic product by $\sh$ and $*$, respectively.
\begin{Lem}[Ihara--Kaneko--Zagier {\cite[Propositions 6 and 7]{ikz}}]\label{lem:ikz}
Define the automorphisms $\Phi_t^{\sh},\Phi_t^*:\HH[[t]]\to \HH[[t]]$ by
\begin{align*}
&\Phi_t^{\sh}(x)=\frac 1{1-yt}x,\qquad \Phi_t^{\sh}(y)=\frac 1{1-yt}y,\qquad \Phi_t^{\sh}(t)=t,\\
&\Phi_t^{*}(x)=x, \qquad\Phi_t^{*}(y)=\frac{y}{1-yt}(1+xt),\qquad\Phi_t^{*}(t)=t.
\end{align*}
Then, for $w\in\HH^1$, the identities
\begin{align*}
w\sh\frac 1{1-yt}&=\Phi^{\sh}_t(w)\frac 1{1-yt},\\
w*\frac 1{1-yt}&=\Phi^{*}_t(w)\frac 1{1-yt}
\end{align*}
hold.
\end{Lem}

\begin{Lem}\label{lem:key}
The following identity holds in $\HH[[t]]$:
\[
\left(\frac 1{1-yx\frac{t^2}{1-xt}}\sh \frac 1{1-yt}\right)*\frac 1{1+yt}=\frac 1{1-y^2xt^3}.
\]
\end{Lem}

\begin{proof}
By Lemma \ref{lem:ikz}, we obtain
\begin{align*}
\left(\frac 1{1-yx\frac{t^2}{1-xt}}\sh\frac 1{1-yt}\right)*\frac 1{1+yt}&=\Phi_{-t}^*\left(\Phi^{\sh}_t\left(\frac 1{1-y\frac{t^2}{1-xt}x}\right)\frac 1{1-yt}\right)\frac 1{1+yt}\\
&=\Phi_{-t}^*\left(\frac 1{1-\frac 1{1-yt}y\frac{t^2}{1-\frac 1{1-yt}xt}\frac 1{1-yt}x}\frac 1{1-yt}\right)\frac 1{1+yt}\\
&=\Phi_{-t}^*\left(\frac 1{1-yt-y\frac{t^2}{1-(x+y)t}x}\right)\frac 1{1+yt}\\
&=\frac 1{1-\frac{y}{1+yt}(1-xt)t-\frac{y}{1+yt}(1-xt)\frac{t^2}{1-\frac 1{1+yt}(x+y)t}x}\frac 1{1+yt}\\
&=\frac 1{1+yt-y(1-xt)t-y(1+yt)xt^2}\\
&=\frac 1{1-y^2xt^3}.
\end{align*}
This completes the proof.
\end{proof}

\begin{proof}[Proof of Theorem \ref{thm:one-two}]
Using Lemma \ref{lem:key} and the shuffle and harmonic relations for finite multiple zeta values (see Kaneko \cite[(8.3) and Theorem 8.1]{k2}), and for symmetric multiple zeta values (see Kaneko \cite[Proposition 9.2 and Theorem 9.6]{k2}), we obtain
\begin{align*}
\sum_{0\leq n}\zeta^{}_{\calF}(\{1,2\}^n)t^{3n}&=\zeta^{}_{\calF}\left(\frac 1{1-y^2xt^3}\right)\\
&=\zeta^{}_{\calF}\left(\left(\frac 1{1-yx\frac{t^2}{1-xt}}\sh \frac 1{1-yt}\right)*\frac 1{1+yt}\right)\\
&=\zeta^{}_{\calF}\left(\frac 1{1-yx\frac{t^2}{1-xt}}\sh \frac 1{1-yt}\right)\\
&=\zeta^{}_{\calF}\left(\frac 1{1-yx\frac{t^2}{1-xt}}\frac 1{1+yt}\right)\\
&=1+\zeta^{}_{\calF}\left(\frac 1{1-yx\frac{t^2}{1-xt}}yx\frac{t^2}{1-xt}\frac 1{1+yt}\right).
\end{align*}

On the other hand, Lemma \ref{lem:truncated} gives
\begin{align*}
&\zeta^{}_{N}\left(\frac 1{1-yx\frac{t^2}{1-xt}}yx\frac{t^2}{1-xt}\frac 1{1+yt}\right)\\
&=t^2\sum_{0\leq m,r}(-1)^rt^{2m+r}\zeta^{}_{N}\left(\left(yx\frac{1}{1-xt}\right)^{m+1}y^r\right)\\
&=t^2\sum_{0\leq m,r}(-1)^rt^{2m+r}\sum_{0<j_1<\cdots<j_m<n<l_1<\cdots<l_r<N}\frac 1{j_1(j_1-t)}\cdots\frac 1{j_m(j_m-t)}\frac 1{n(n-t)}\frac 1{l_1\cdots l_r}\\
&=t^2\sum_{0<n<N}\left(\prod_{j=1}^{n-1}\left(1+\frac{t^2}{j(j-t)}\right)\right)\frac 1{n(n-t)}\prod_{j=n+1}^{N-1}\left(1-\frac tj\right)\\
&=\sum_{0<n<N}\frac{(\omega t,\omega^2 t)_{n}}{(1-t)_n^2}\frac{(1-t)_{N-1}}{(N-1)!}\\
&=t^2\zeta^{}_{N}\left(\frac 1{1+yt-yxt^2}yx\frac 1{1+2yt-yxt^2}\right)\zeta^{}_{N}\left(\frac 1{1+yt}\right)^{-1},
\end{align*}
where $\omega\coloneqq e^{\frac{2\pi i}3}$. Hence, we have
\[
\zeta^{}_{N}\left(\frac 1{1-yx\frac{t^2}{1-xt}}yx\frac{t^2}{1-xt}\frac 1{1+yt}*\frac 1{1+yt}\right)=t^2\zeta^{}_{N}\left(\frac 1{1+yt-yxt^2}yx\frac 1{1+2yt-yxt^2}\right).
\]

By the linear independence of multiple harmonic sums (see Yamamoto \cite[Theorem 3.1]{y}), we have
\[
\frac 1{1-yx\frac{t^2}{1-xt}}yx\frac{t^2}{1-xt}\frac 1{1+yt}*\frac 1{1+yt}=t^2\frac 1{1+yt-yxt^2}yx\frac 1{1+2yt-yxt^2}.
\]
Hence, 
\begin{align*}\label{eq:toKaneko}
1+\zeta^{}_{\calF}\left(\frac 1{1-yx\frac{t^2}{1-xt}}yx\frac{t^2}{1-xt}\frac 1{1+yt}\right)&=1+t^2\zeta^{}_{\calF}\left(\frac 1{1+yt-yxt^2}yx\frac 1{1+2yt-yxt^2}\right).
\end{align*}
Put $a=\omega t, b=\omega^2t, c=d=t$ in Theorem \ref{thm:main}. 
We have
\begin{align*}
&1+t^2\zeta^{}_{\calF}\left(\frac 1{1+yt-yxt^2}yx\frac 1{1+2yt-yxt^2}\right)\\
&=\exp\biggl(3\sum_{2\le n}\frac{Z_{\calF}(n)}{n}(\omega^{2n}+\omega^n+1)t^{n}\biggr)\\
&=\exp\biggl(3\sum_{0<n}\frac{Z_{\calF}(3n)}{n}t^{3n}\biggr).
\end{align*}
Note that for the case $\calF=\calA$, we consider the identity above in the ring $(\calZ_{\calA}\otimes_{\QQ}\QQ(\omega))[[t]]$.

Kaneko \cite[Example 7.4 and 9.4]{k2} gives
\[
\zeta^{}_{\calF}(1,3n-1)=3nZ_{\calF}(3n). \label{eq:FDZV}
\]
Therefore, replacing $t^3$ by $t$, we obtain Theorem \ref{thm:one-two}.
\end{proof}

\section*{Acknowledgements}
The second author was supported by JSPS KAKENHI Grant Number
JP23K12962.


\begin{thebibliography}{KMY}
\bibitem[D]{d} J. Dougall, \emph{On Vandermonde's theorem, and some more general expansions}, Proc. Edinburgh Math. Soc. \textbf{25} (1907), 114--132.
\bibitem[HHT]{hht} Kh. Hessami Pilehrood, T. Hessami Pilehrood, R. Tauraso, \emph{New properties of multiple harmonic sums modulo $p$ and $p$-analogues of Leshchiner’s series}, Trans.\ Amer.\ Math.\ Soc.\ \textbf{366} (2014), 3131--3159.
\bibitem[IKZ]{ikz} K. Ihara, M. Kaneko and D. Zagier, \emph{Derivation and double shuffle relations for multiple zeta values}, Compos.
Math. \textbf{142} (2006), no. 2, 307--338.
\bibitem[K]{k} M. Kaneko, \emph{Finite multiple zeta values} (in Japanese), RIMS Kôkyûroku Bessatsu B68 (2017), 175--190.
\bibitem[K2]{k2} M. Kaneko, \emph{An introduction to classical and finite multiple zeta values}, Publ. Math. Besançon (2019), no. 1,  103--129.
\bibitem[KMY]{kmy} M. Kaneko, T. Murakami and A. Yoshihara, \emph{On finite multiple zeta values of level two}, Pure Appl. Math. Q. \textbf{19} (2023), 267--280.
\bibitem[KZ]{kz} M. Kaneko and D. Zagier, \emph{Finite multiple zeta values}, to appear in the Proceedings of the 17th MSJ-SI Conference on Modular Forms and Multiple Zeta Values.
\bibitem[KMO]{kmo}
H.~Kawamura, T.~Maesaka and M.~Ono,
\emph{Weighted sum formula for variants of half multiple zeta values}, Int. J. Number Theory \textbf{20} (2024),no.~9, 2329--2353.
\bibitem[M]{m} H. Murahara, \emph{A note on finite real multiple zeta values}, Kyushu J. Math. \textbf{70} (2016), 197--204.
\bibitem[MS]{ms} H. Murahara and S. Saito, \emph{Restricted sum formula for finite and symmetric multiple zeta values}, Pacific J. Math. \textbf{303} (2019), 325--335.
\bibitem[OZ]{oz}
Y.~Ohno and D.~Zagier, \emph{Multiple zeta values of fixed weight, depth, and height}, Indag. Math. (N.S.) \textbf{12} (2001), no.~4, 483--487.
\bibitem[SW]{sw} S. Saito and N. Wakabayashi, \emph{Sum formula for finite multiple zeta values}, J. Math. Soc. Japan \textbf{67} (2015), 1069--1076.
\bibitem[W]{w} Z. Wang, \emph{Refined sum formulas for finite multiple zeta values of level two}, preprint, arXiv:2609.23562.
\bibitem[Y]{y} S. Yamamoto, \emph{Explicit evaluation of certain sums of multiple zeta-star values}, Funct. Approx. Comment. Math. \textbf{49} (2013), 283--289.

\end{thebibliography}
\end{document}